\documentclass[12pt, letterpaper]{amsart}
\usepackage{amsfonts,amsthm,mathtools}
\usepackage[all,cmtip]{xy}
\usepackage{fullpage}
\usepackage{amsmath}
\usepackage{amssymb}
\usepackage{enumerate}
\usepackage{tikz}
\usepackage[backref]{hyperref}
\usepackage{cleveref}
\usepackage{verbatim}
\usepackage{cancel}
\usepackage{float}
\usepackage{longtable}
\usepackage{listings}

\newtheorem{lem}{Lemma}[section]
\newtheorem{thm}[lem]{Theorem}

\newtheorem{prop}[lem]{Proposition}
\newtheorem{cor}[lem]{Corollary}
\newtheorem{conj}[lem]{Conjecture}

\theoremstyle{definition}
\newtheorem{remark}[lem]{Remark}
\newtheorem{definition}[lem]{Definition}

\DeclareMathAlphabet{\curly}{U}{rsfs}{m}{n}

\newcommand{\gon}{\operatorname{gon}}

\newcommand{\Q}{\mathbb{Q}}
\newcommand{\C}{\mathbb{C}}

\newcommand{\Z}{\mathbb{Z}}
\newcommand{\F}{\mathbb{F}}

\newcommand{\SL}{\operatorname{SL}}

\newcommand{\PP}{{\mathbb P}}

\mathchardef\mhyphen="2D

\title{Tetragonal modular quotients $X_0^{+d}(N)$}

\author{\sc Petar Orli\'c}
\address{Petar Orli\'c \\
University of Zagreb\\  
Bijeni\v{c}ka Cesta 30 \\
10000 Zagreb\\
Croatia}
\email{petar.orlic@math.hr}

\begin{document}
\begin{abstract}
    Let $N$ be a positive integer. For every $d\mid N$ such that $(d,N/d)=1$ there exists an Atkin-Lehner involution $w_d$ of the modular curve $X_0(N)$. In this paper we determine all quotient curves $X_0(N)/w_d$ whose $\Q$-gonality is equal to $4$ and all quotient curves $X_0(N)/w_d$ whose $\C$-gonality is equal to $4$.
\end{abstract}

\subjclass{11G18, 11G30, 14H30, 14H51}
\keywords{Modular curves, Gonality}

\thanks{The author was supported by the Croatian Science Foundation under the project no. IP-2022-10-5008.}

\maketitle

\section{Introduction}
Let $C$ be a smooth projective curve over a field $k$. The $k$-gonality of $C$, denoted by $\textup{gon}_k C$, is the least degree of a non-constant $k$-rational morphism $f:C\to\mathbb{P}^1$. For curves of genus $g\geq2$ there exists an upper bound for $\textup{gon}_k C$, linear in terms of the genus, see \Cref{poonen}.

When $C$ is a modular curve, there also exists a linear lower bound for the $\C$-gonality. This was first proved by Zograf \cite{Zograf1987}. The constant was afterwards improved by Abramovich \cite{abramovich} and later by Kim and Sarnak in Appendix 2 to \cite{Kim2002}.

The gonality of the modular curve $X_0(N)$ and its quotients has been extensively studied over the years. Ogg \cite{Ogg74} determined all hyperelliptic curves $X_0(N)$, Bars \cite{Bars99} determined all bielliptic curves $X_0(N)$, Hasegawa and Shimura determined all trigonal curves $X_0(N)$ over $\C$ and $\Q$, Jeon and Park determined all tetragonal curves $X_0(N)$ over $\C$, and Najman and Orlić \cite{NajmanOrlic22} determined all curves $X_0(N)$ with $\Q$-gonality equal to $4$, $5$, or $6$, and also determined the $\Q$ and $\C$-gonality for many other curves $X_0(N)$.

Regarding the gonality of the quotients of the curve $X_0(N)$, Furumoto and Hasegawa \cite{FurumotoHasegawa1999} determined all hyperelliptic quotients, and Hasegawa and Shimura \cite{HasegawaShimura1999,HasegawaShimura2000,HasegawaShimura2006} determined all trigonal quotients over $\C$. Bars, Gonzalez, and Kamel \cite{BARS2020380} determined all bielliptic quotients of $X_0(N)$ for squarefree levels $N$, Jeon \cite{JEON2018319} determined all bielliptic quotients $X_0^+(N)$, and Bars, Kamel, and Schweizer \cite{bars22biellipticquotients} determined all bielliptic quotients of $X_0(N)$ for non-squarefree levels $N$, completing the classification of bielliptic quotients.

The next logical step is to determine all tetragonal quotients of $X_0(N)$. All tetragonal quotients $X_0^+(N)$ over $\C$ and $\Q$ were determined in \cite{OrlicTetragonalX0+}. Here, we will do the same for the quotients $X_0^{+d}(N):=X_0(N)/w_d$ for $d\neq N$ (the notation $X_0^{+d}(N)$ was taken from \cite{HasegawaShimura1999}). We also determine all curves $X_0^{+d}(N)$ of genus $4$ that are trigonal over $\Q$, thus completing the classification of all $\Q$-trigonal curves $X_0^{+d}(N)$, mainly done by Hasegawa and Shimura \cite{HasegawaShimura1999}.

Our main results are the following theorems. For expository reasons, we also include the previously solved case $d=N$.

\begin{thm}\label{trigonalthm}
    The curve $X_0^{+d}(N):=X_0(N)/w_d$ is of genus $4$ and has $\Q$-gonality equal to $3$ if and only if
    \begin{align*}
        (N,d)\in\{&(66,33),(74,37),(84,84),(86,43),(88,88),(93,93),(108,4),(112,7),\\
        &(115,115),(116,116),(129,129),(135,135),(137,137),(147,147),\\
        &(155,155),(159,159),(215,215)\}.
    \end{align*}
\end{thm}

\begin{thm}\label{tetragonalthm}
    The curve $X_0^{+d}(N):=X_0(N)/w_d$ has $\Q$-gonality equal to $4$ if and only if the pair $(N,d)$ is in the following table. In all cases when the genus of the curve $X_0^{+d}(N)$ is not $4$ (all genus $4$ cases are listed in \Cref{Fp_trigonal}), the $\C$-gonality is also equal to $4$.

    Additionally, for $N=243,271$, the curve $X_0^+(N)$ is tetragonal over $\C$, but not over $\Q$.

    \begin{center}
\hspace*{-1cm}\begin{tabular}{|c|c||c|c||c|c||c|c|}
\hline
    $N$ & $d$ & $N$ & $d$ & $N$ & $d$ & $N$ & $d$\\
    \hline
    $60$ & $3$, $5$ &
    $66$ & $2$, $3$, $22$ &
    $68$ & $17$ &
    $70$ & $2$, $5$, $7$, $70$\\
    $74$ & $2$ &
    $76$ & $4$ &
    $77$ & $11$ &
    $78$ & $2$, $3$, $6$, $13$, $78$\\
    $80$ & $5$ &
    $82$ & $2$, $82$ &
    $84$ & $3,4,7,12,21,28$ &
    $85$ & $5$, $17$\\
    $88$ & $8$, $11$ &
    $90$ & $2,5,9,10,18,45,90$ &
    $91$ & $7$ &
    $93$ & $3$, $31$ \\
    $96$ & $3$ &
    $98$ & $2$ &
    $99$ & $9$ &
    $100$ & $25$ \\
    $102$ & $2$, $3$, $17$, $51$, $102$ &
    $104$ & $8$, $13$ &
    $105$ & $3,5,7,15,21,35,105$ &
    $106$ & $2$,$53$,$106$ \\
    $108$ & $27$, $108$ &
    $110$ & $2,5,10,11,22,55,110$ &
    $111$ & $3$, $37$ &
    $112$ & $16$, $112$ \\
    $114$ & $2$, $3$, $19$, $38$, $114$ &
    $115$ & $5$, $23$ &
    $116$ & $4$, $29$ &
    $117$ & $9$, $117$ \\
    $118$ & $2$, $59$, $118$ &
    $120$ & $5$, $8$, $15$, $24$, $40$, $120$ &
    $122$ & $2$, $61$ &
    $123$ & $3$, $41$, $123$ \\
    $124$ & $4$, $31$, $124$ &
    $126$ & $2,7,9,14,18,63,126$ &
    $129$ & $3$, $43$ &
    $130$ & $2$, $10$, $13$, $26$, $65$ \\
    $132$ & $4$, $11$, $44$ &
    $133$ & $19$, $133$ &
    $134$ & $2$, $67$, $134$ &
    $135$ & $5$, $27$ \\
    $136$ & $8$, $17$, $136$ &
    $138$ & $3$, $6$, $23$, $69$, $138$ &
    $140$ & $4$, $35$, $140$ &
    $141$ & $3$, $47$, $141$ \\
    $142$ & $2$, $71$, $142$ &
    $143$ & $11$, $13$ &
    $144$ & $9$, $16$, $144$ &
    $145$ & $29$, $145$ \\
    $146$ & $2$, $73$ &
    $147$ & $49$ &
    $148$ & $4$, $148$ &
    $152$ & $152$ \\
    $153$ & $9$, $17$ &
    $155$ & $5$, $31$ &
    $156$ & $4$, $39$, $156$ &
    $157$ & $157$ \\
    $158$ & $2$, $79$, $158$ &
    $159$ & $3$, $53$ &
    $160$ & $32$, $160$ &
    $161$ & $7$, $23$, $161$ \\
    $163$ & $163$ &
    $165$ & $11$, $15$, $165$ &
    $166$ & $2$, $83$, $166$ &
    $168$ & $21$, $24$, $56$ \\
    $171$ & $9$, $19$, $171$ &
    $173$ & $173$ &
    $175$ & $175$ &
    $176$ & $11$, $16$, $176$ \\
    $177$ & $3$, $59$, $177$ &
    $183$ & $183$ &
    $184$ & $8$, $23$, $184$ &
    $185$ & $185$ \\
    $188$ & $4$, $47$, $188$ &
    $190$ & $5$, $10$, $19$, $95$ &
    $192$ & $192$ &
    $193$ & $193$ \\
    $194$ & $194$ &
    $195$ & $5$, $39$, $195$ &
    $196$ & $4$ &
    $197$ & $197$ \\
    $199$ & $199$ &
    $200$ & $200$ &
    $203$ & $203$ &
    $205$ & $5$, $41$, $205$ \\
    $206$ & $2$, $103$, $206$ &
    $207$ & $9$, $23$, $207$ &
    $209$ & $11$, $19$, $209$ &
    $211$ & $211$ \\
    $213$ & $3$, $71$, $213$ &
    $215$ & $5$, $43$ &
    $221$ & $13$, $17$, $221$ &
    $223$ & $223$ \\
    $224$ & $224$ &
    $229$ & $229$ &
    $241$ & $241$ &
    $251$ & $251$ \\
    $257$ & $257$ &
    $263$ & $263$ &
    $269$ & $269$ &
    $279$ & $9$, $31$, $279$ \\
    $281$ & $281$ &
    $284$ & $4$, $71$, $284$ &
    $287$ & $7$, $41$, $287$ &
    $299$ & $13$, $23$, $299$ \\
    $311$& $311$ &
    $359$ & $359$ & & & & \\
    
    \hline
\end{tabular}
\end{center}
\end{thm}

Interestingly, as we can see, it turns out that all curves $X_0^{+d}(N)$ for $d\neq N$ that are $\C$-tetragonal are also $\Q$-tetragonal. However, this property does not hold for curves $X_0^+(N)$ since for $N=243,271$ the curve $X_0^+(N)$ is $\C$-tetragonal, but not $\Q$-tetragonal.

We use similar methods to the ones used in \cite{OrlicTetragonalX0+} to determine the tetragonal curves $X_0^{+d}(N)$. In \Cref{Fpsection}, we give lower bounds to the $\Q$-gonality by computing the gonality over finite fields. In \Cref{degree4mapsection}, we construct degree $4$ rational morphisms to $\mathbb{P}^1$, either via quotient maps to curves $X_0(N)/\left<w_d,w_{d'}\right>$ or using \texttt{Magma}. In \Cref{C-gonality_section}, we give lower bounds to the $\C$-gonality. For some curves we use the Castelnuovo-Severi inequality (\Cref{tm:CS}). We also use the graded Betti numbers to disprove the existence of degree $4$ morphisms to $\mathbb{P}^1$.

Note that for each level $N$ that is not a prime power, there are multiple quotients $X_0^{+d}(N)$ that need to be checked. For example, for $N=210$ which has four different prime factors, there are $14$ such quotients (we can exclude the curve $X_0^+(210)$ because it has already been solved). Therefore, it can be hard to track whether all quotients have been solved.

For the reader's convenience, at the end of the paper we put Tables \ref{tab:main1}, 
\ref{tab:main2}, \ref{tab:main3}. In these tables, for each level $N$, we give the links to all propositions used to solve the quotients at that level.

A lot of the results in this paper rely on \texttt{Magma} computations. The version of \texttt{Magma} used in the computations is V2.28-15, the latest version at the time of the writing of this paper. The codes that verify all computations in this paper can be found on
\begin{center}
    \url{https://github.com/orlic1/gonality_X0_quotients}.
\end{center}
All computations were performed on the Euler server at the Department of Mathematics, University of Zagreb with a Intel Xeon W-2133 CPU running at 3.60GHz and with
64 GB of RAM.

\section*{Conflicts of Interest}

The author has no relevant financial or non-financial interests to disclose.

\section*{Acknowledgements}

Many thanks to Nikola Adžaga and Philippe Michaud-Jacobs for permission to use their \texttt{Magma} codes as templates. Additionally, code and data associated to the paper \cite{Rouse_Sutherland_Zureick-Brown_2022} by Jeremy Rouse, Andrew V. Sutherland, and David Zureick-Brown was used in \Cref{Fp2points_prop}. Their code can be found on
\begin{center}
    \url{https://github.com/AndrewVSutherland/ell-adic-galois-images/tree/209c2f888669785151174f472ea2c9eafb6daaa9}.
\end{center}
The code for computations in \Cref{Fp2points_prop_quotient} uses functions from Francesc Bars's repository
\begin{center}
    \url{https://github.com/FrancescBars/Magma-functions-on-Quotient-Modular-Curves}.
\end{center}

I am also grateful to Filip Najman and Maarten Derickx for their helpful comments and suggestions. Finally, I would like to thank the anonymous referees for their reviews which have greatly improved the paper.

\section{Lower bounds on $\Q$-gonality via $\F_p$-gonality}\label{Fpsection}
For a curve $C$ defined over $\Q$ and $p$ a prime of good reduction of $C$, it is known that
$$\textup{gon}_{\F_p}(C)\leq\textup{gon}_\Q(C).$$
This is an important tool for determining the $\Q$-gonality because it is generally much easier to find the gonality over finite fields. When working with the modular curve $X_0(N)$ and its quotients, a very helpful result is Ogg's inequality \cite[Theorem 3.1]{Ogg74}, stated in simpler form in \cite[Lemma 3.1]{HasegawaShimura_trig}.

\begin{lem}\label{lemmaogg}
    Let $p$ be a prime not dividing $N$. Then the number of $\F_{p^2}$-points on the curve $X_0(N)$ is at least
    $$L_p(N):=\frac{p-1}{12}\psi(N)+2^{\omega(N)}.$$
    Here $\psi(N)$ is the index of the congruence subgroup $\Gamma_0(N)$, equal to $N\prod_{q\mid N}(1+\frac{1}{q})$, and $\omega(N)$ is the number of different prime divisors of $N$.
\end{lem}
\begin{lem}\cite[Lemma 3.5]{NajmanOrlic22}\label{Fp2points}
    Let $C$ be a curve over $\Q$, $p$ a prime of good reduction for $C$, and $q$ a power of $p$. Suppose $\#C(\F_q)>d(q+1)$. Then $\textup{gon}_\Q(C)>d$.
\end{lem}

If the quotient curve $X_0^{+d}(N)$ is tetragonal, then we have a rational composition map $X_0(N)\to X_0^{+d}(N)\to\mathbb{P}^1$ of degree $8$. Therefore, by \Cref{Fp2points} we must have
\begin{align}\label{ogg}
    L_p(N)&\leq8(p^2+1)
\end{align}
for all primes $p\nmid N$. Also, notice that we can exclude all levels $N$ which are prime powers since in that case the only quotient of $X_0(N)$ is the curve $X_0^+(N)$ which has been solved in \cite{OrlicTetragonalX0+}. From now on, we will suppose that $N$ is not a prime power.

\begin{prop}
    For every $N>432$ that is not a prime power there exists a prime $p$ for which the inequality \ref{ogg} does not hold.
\end{prop}

\begin{proof}
    The proof is similar to the proof of \cite[Lemma 3.2]{HasegawaShimura_trig}. Since $N$ is not a prime power, we have $\omega(N)\geq2$. Now there are several cases:
    \begin{itemize}
        \item $2\nmid N, N>432$: take $p=2$; $\displaystyle\frac{p-1}{12}\psi(N)+2^{\omega(N)}\geq\frac{N}{12}+4>40=8(p^2+1)$.
        \item $2\mid N, 3\nmid N, N>304$: take $p=2$; $\displaystyle\frac{p-1}{12}\psi(N)+2^{\omega(N)}\geq\frac{N}{4}+4>80=8(p^2+1)$.
        \item $(2\cdot3)\mid N, 5\nmid N, N>306$: take $p=5$; $\displaystyle\frac{p-1}{12}\psi(N)+2^{\omega(N)}\geq\frac{2N}{3}+4>208=8(p^2+1)$.
        \item $(2\cdot3\cdot5)\mid N, 7\nmid N, N>326$: take $p=7$; $\displaystyle\frac{p-1}{12}\psi(N)+2^{\omega(N)}\geq\frac{6N}{5}+8>400=8(p^2+1)$.
        \item $(2\cdot3\cdot5\cdot7)\mid N, 11\nmid N, N>420$: take $p=11$; $\displaystyle\frac{p-1}{12}\psi(N)+2^{\omega(N)}\geq\frac{16N}{7}+16>976=8(p^2+1)$.
        \item $(2\cdot3\cdot5\cdot7\cdot11)\mid N$: take $p$ to be the smallest prime not dividing $N$ and let $q$ be the largest prime dividing $N$. Now we have \[\frac{p-1}{12}\psi(N)+2^{\omega(N)}\geq\frac{p-1}{12}\cdot\frac{3\cdot4\cdot6\cdot8\cdot12}{2\cdot3\cdot5\cdot7\cdot11}N+32\geq\frac{576(p-1)q}{11}+32.\]
        In the last inequality we used that $N\geq2\cdot3\cdot5\cdot7\cdot q$. Since $p<2q$ by Bertrand's postulate, it is now elementary to prove that this is greater than $8(p^2+1)$.
    \end{itemize}
\end{proof}

As we can see from the proof, this result also eliminates many levels $N$ smaller than $432$. More precisely, using inequality \ref{ogg} we can eliminate 
\begin{align*}
    N\in\{&255,260,266,273,276,280,282,285,286,290,292,294,296,304,\\
    &306,308,310-318,320,322,324,326-328,330-334,336-340,\\
    &342-354,356-370,372-376,378-390,392-420,422-430,432\}.
\end{align*}

The following results are a direct application of \Cref{Fp2points}.

\begin{prop}\label{Fp2points_prop}
    The curve $X_0^{+d}(N)$ is not tetragonal over $\Q$ for $N=420$ and all $15$ possible values of $d$.
\end{prop}

\begin{proof}
    Using \texttt{Magma}, we calculate that the curve $X_0(420)$ has $1128$ points over $\F_{11^2}$. The code is available on Github and the running time is around $1$ minute. \Cref{Fp2points} now tells us that the $\Q$-gonality of the curve $X_0(420)$ is at least $10$. Therefore, the $\Q$-gonality of all quotient curves $X_0^{+d}(420)$ is at least $5$.
\end{proof}

\begin{prop}\label{Fp2points_prop_quotient}
    The curve $X_0^{+d}(N)$ is not tetragonal over $\Q$ for the following values of $N$ and $d$:

        \begin{center}
\begin{tabular}{|c|c|c||c|c|c||c|c|c|}
\hline
    $(N,d)$ & $q$ & $\#X_0^{+d}(\F_q)$ & $(N,d)$ & $q$ & $\#X_0^{+d}(\F_q)$ & $(N,d)$ & $q$ & $\#X_0^{+d}(\F_q)$\\
    \hline
    $(140,5)$ & $9$ & $42$ & $(165,3)$ & $4$ & $24$ & $(195,13)$ & $4$ & $22$\\
    $(200,25)$ & $9$ & $42$ & $(208,13)$ & $25$ & $122$ & $(212,53)$ & $3$ & $18$\\
    $(220,55)$ & $9$ & $54$ & $(224,7)$ & $9$ & $44$ & $(225,25)$ & $4$ & $27$\\
    $(226,2)$ & $9$ & $42$ & $(237,3)$ & $4$ & $23$ & $(242,2)$ & $9$ & $45$\\
    $(242,121)$ & $9$ & $43$ & $(254,127)$ & $9$ & $46$ & $(259,7)$ & $4$ & $22$\\
    $(261,29)$ & $4$ & $22$ & $(268,4)$ & $9$ & $52$ & $(274,137)$ & $9$ & $42$\\
    $(275,11)$ & $9$ & $48$ & $(278,139)$ & $9$ & $51$ & $(288,9)$ & $25$ & $128$\\
    $(297,11)$ & $4$ & $25$ & $(298,149)$ & $9$ & $42$ & $(301,7)$ & $4$ & $24$\\
    $(302,151)$ & $9$ & $54$ & $(323,19)$ & $4$ & $22$ & $(325,25)$ & $4$ & $25$\\
    $(355,71)$ & $9$ & $42$ & & & & & &\\
    \hline
\end{tabular}
\end{center}
\end{prop}

\begin{proof}
    In all these cases we use \texttt{Magma} to compute the number of $\F_q$-rational points on $X_0^{+d}(N)$ and it turns out to be greater than $4(q+1)$. Therefore, the $\Q$-gonality of these curves is at least $5$.
\end{proof}

In the following results we will use Poonen's \cite[Proposition A.1.]{Poonen2007}.

\begin{prop}\label{poonen}
    Let $X$ be a curve of genus $g$ over a field $k$.
    \begin{enumerate}[(i)]
        \item If $L$ is a field extension of $k$, then $\textup{gon}_L(X)\leq \textup{gon}_k(X)$.
        \item If $k$ is algebraically closed and $L$ is a field extension of $k$, then $\textup{gon}_L(X)=\textup{gon}_k(X)$.
        \item If $g\geq2$, then $\textup{gon}_k(X)\leq 2g-2$.
        \item If $g\geq2$ and $X(k)\neq\emptyset$, then $\textup{gon}_k(X)\leq g$.
        \item If $k$ is algebraically closed, then $\textup{gon}_k(X)\leq\frac{g+3}{2}$.
        \item If $\pi:X\to Y$ is a dominant $k$-rational map, then $\textup{gon}_k(X)\leq \deg \pi\cdot\textup{gon}_k(Y)$.
        \item If $\pi:X\to Y$ is a dominant $k$-rational map, then $\textup{gon}_k(X)\geq\textup{gon}_k(Y)$.
    \end{enumerate}
\end{prop}

Since all modular curves $X_0(N)$ and their quotients have at least one rational cusp and we are interested only in those curves of genus $g\geq2$, this result implies that their $\Q$-gonalities are bounded from above by their genus.

We can also compute the $\F_p$-gonality of a curve defined over $\Q$ by checking the dimensions of Riemann-Roch spaces of degree $d$ effective $\F_p$-rational divisors.

\begin{prop}\label{Fp_trigonal}
    The $\Q$-gonality of the genus $4$ curve $X_0^{+d}(N)$ is equal to $4$ for the following values of $N$ and $d$:

        \begin{center}
\begin{tabular}{|c|c||c|c||c|c||c|c|}
\hline
    $(N,d)$ & $p$ & $(N,d)$ & $p$ & $(N,d)$ & $p$ & $(N,d)$ & $p$\\
    \hline
    
$(60,3)$ & $7$ & $(60,5)$ & $7$ & $(66,2)$ & $13$ & $(68,17)$ & $3$\\ $(70,5)$ & $17$ & $(74,2)$ & $3$ & $(76,4)$ & $5$ & $(77,11)$ & $3$\\
$(80,5)$ & $7$ & $(82,2)$ & $7$ & $(85,5)$ & $23$ & $(85,17)$ & $7$\\
$(88,8)$ & $5$ & $(91,7)$ & $11$ & $(93,3)$ & $5$ & $(98,2)$ & $11$\\
$(100,25)$ & $3$ & $(108,27)$ & $5$ & $(110,55)$ & $7$ & $(133,19)$ & $5$\\
$(145,29)$ & $11$ & $(177,59)$ & $5$ & $(188,47)$ & $3$ & &\\
    
    \hline
\end{tabular}
\end{center}
\end{prop}

\begin{proof}
    Using \texttt{Magma}, we compute that there are no functions of degree $\leq3$ in $\F_p(X_0^{+d}(N))$. To do this, we check that every degree $3$ effective $\F_p$-rational divisor has Riemann-Roch dimension equal to $1$ (the costant function is always in the Riemann-Roch space).
    
    We can reduce the number of divisors that need to be checked by noting the following: If there exists a function $f$ over a field $k$ of a certain degree and if $c=f(x_0)$ for some $k$-rational point $x_0$, then the function $g(x):=\frac{1}{f(x)-c}$ has the same degree and its polar divisor contains a $k$-rational point. Therefore, we only need to consider divisors of the form \[1+1+1 \textup{ or } 1+2,\] i.e. sums of three $\F_p$-rational points or an $\F_p$-rational point + a quadratic point and its conjugate. We obtain these points as places in the algorithmic function field of the curve $X^{+d}(N)$ over $\F_p$. See the codes on Github for more details.

    On the other hand, all these curves are of genus $4$ and have at least one rational cusp. By \Cref{poonen}, this implies that their $\Q$-gonality is at most $4$.
\end{proof}

\begin{prop}\label{Fp_gonality}
The $\Q$-gonality of the curve $X_0^{+d}(N)$ at least $5$ for the following values of $N$ and $d$:

\begin{center}
\begin{tabular}{|c|c||c|c||c|c||c|c||c|c|}
\hline
$(N,d)$ & $p$ & $(N,d)$ & $p$ & $(N,d)$ & $p$ & $(N,d)$ & $p$ & $(N,d)$ & $p$\\
    \hline
    
$(132,33)$ & $5$ & $(140,28)$ & $3$ & $(150,25)$ & $7$ & $(154,2)$ & $3$ & $(154,7)$ & $3$\\
$(154,11)$ & $3$ & $(154,14)$ & $5$ & $(154,22)$ & $5$ & $(154,77)$ & $3$ & $(164,41)$ & $3$\\
$(165,5)$ & $2$ & $(165,33)$ & $7$ & $(165,55)$ & $2$ & $(168,3)$ & $5$ & $(168,7)$ & $5$\\
$(168,8)$ & $5$ & $(170,2)$ & $7$ & $(170,5)$ & $3$ & $(170,10)$ & $3$ & $(170,17)$ & $3$\\
$(170,34)$ & $3$ & $(170,85)$ & $3$ & $(172,4)$ & $3$ & $(180,4)$ & $7$ & $(180,5)$ & $7$\\
$(180,9)$ & $7$ & $(180,20)$ & $7$ & $(180,36)$ & $7$ & $(180,45)$ & $7$ & $(186,2)$ & $5$\\
$(186,31)$ & $5$ & $(186,62)$ & $5$ & $(192,3)$ & $5$ & $(192,64)$ & $5$ & $(198,11)$ & $5$\\
$(198,22)$ & $5$ & $(198,99)$ & $5$ & $(200,8)$ & $3$ & $(201,3)$ & $2$ & $(201,67)$ & $2$\\
$(204,68)$ & $5$ & $(210,35)$ & $11$ & $(212,4)$ & $3$ & $(216,8)$ & $5$ & $(216,27)$ & $5$\\
$(218,2)$ & $5$ & $(218,109)$ & $5$ & $(219,3)$ & $5$ & $(219,73)$ & $2$ & $(226,113)$ & $3$\\
$(232,8)$ & $3$ & $(232,29)$ & $3$ & $(234,13)$ & $5$ & $(234,26)$ & $5$ & $(234,117)$ & $5$\\
$(235,5)$ & $7$ & $(235,47)$ & $3$ & $(237,79)$ & $7$ & $(240,80)$ & $11$ & $(244,4)$ & $3$\\
$(244,61)$ & $3$ & $(247,13)$ & $5$ & $(247,19)$ & $5$ & $(250,125)$ & $3$ & $(252,63)$ & $5$\\
$(253,11)$ & $2$ & $(253,23)$ & $3$ & $(258,86)$ & $5$ & $(265,5)$ & $3$ & $(265,53)$ & $3$\\
$(268,67)$ & $3$ & $(272,16)$ & $3$ & $(288,32)$ & $5$ & $(291,3)$ & $5$ & $(301,43)$ & $3$\\
    \hline
\end{tabular}
\end{center}
\end{prop}

\begin{proof}
    Similarly as in the previous proposition, we use \texttt{Magma} to compute that there are no functions of degree $\leq4$ in $\F_p(X_0^{+d}(N))$. Here we need to consider divisors of the form \[1+1+1+1, 1+1+2, \textup{ or } 1+3.\] The argumentation for that and the method of obtaining all such divisors is the same as in \Cref{Fp_trigonal}.
\end{proof}

Some computations in \Cref{Fp_gonality} were running for more than an hour, especially in the higher genus cases. This approach is not feasible for the curves in the next proposition which are all of high genus. For example, the curve $X_0^{+6}(246)$ is of genus $20$ and the curve $X_0^{+3}(300)$ is of genus $22$. Instead, we can prove that the quotient curve $X_0(N)/\left<w_d,w_{d'}\right>$, which is of smaller genus, is not tetragonal.

\begin{prop}\label{Fp_gonality_quotient}
The $\Q$-gonality of the curve $X_0^{+d}(N)$ is at least $5$ for the following values of $N$ and $d$:

\begin{center}
\begin{tabular}{|c|c|c|c||c|c|c|c|}
\hline
$N$ & $d$ & $p$ & $Y$ & $N$ & $d$ & $p$ & $Y$ \\
    \hline

$228$ & $3$, $19$, $57$ & $5$ & $X_0(228)/\left<w_3,w_{19}\right>$ & $228$ & $12$, $57$, $76$ & $5$ & $X_0(228)/\left<w_{12},w_{57}\right>$\\
$228$ & $4$, $57$ & $5$ & $X_0(228)/\left<w_4,w_{57}\right>$ & $240$ & $3$, $5$, $15$ & $11$ & $X_0(240)/\left<w_3,w_5\right>$\\
$240$ & $3$, $16$, $48$ & $11$ & $X_0(240)/\left<w_3,w_{16}\right>$ & $246$ & $6$, $82$, $123$ & $5$ & $X_0(246)/\left<w_6,w_{82}\right>$\\
$246$ & $3,82$ & $5$ & $X_0(246)/\left<w_3,w_{82}\right>$ &
$264$ & $3$, $8$, $24$ & $5$ & $X_0(264)/\left<w_3,w_8\right>$\\
$264$ & $3$, $11$, $33$ & $5$ & $X_0(264)/\left<w_3,w_{11}\right>$ & $264$ & $8$, $11$, $88$ & $5$ & $X_0(264)/\left<w_8,w_{11}\right>$\\
$270$ & $5$, $27$, $135$ & $7$ & $X_0(270)/\left<w_5,w_{27}\right>$ & $300$ & $3$, $100$ & $7$ & $X_0(300)/\left<w_3,w_{100}\right>$\\
$300$ & $12$, $25$ & $7$ & $X_0(300)/\left<w_{12},w_{25}\right>$ & $309$ & $3$, $103$ & $5$ & $X_0(309)/\left<w_3,w_{103}\right>$\\

    \hline
\end{tabular}
\end{center}
\end{prop}

\begin{proof}
    Using \texttt{Magma}, we compute that there are no $\F_p$-rational functions of degree $\leq4$ from $Y=X_0(N)/\left<w_d,w_{d'}\right>$ to $\mathbb{P}^1$. Since there is a rational degree $2$ quotient map $X_0^{+d}(N)\to Y$, \Cref{poonen}(vii) tells us that the $\F_p$-gonality of $X_0^{+d}(N)$ is $\geq5$.
\end{proof}

%In \Cref{bettisection}, we will see that all curves of genus at least $10$ that are not $\Q$-tetragonal are also not $\C$-tetragonal. Furthermore, all such curves have been determined in this section.

\section{Rational morphisms to $\mathbb{P}^1$}\label{degree4mapsection}

In most cases, when there exists a degree $4$ rational morphism $X_0^{+d}(N)\to \mathbb{P}^1$, we can realise it via the rational quotient map to the curve $X_0(N)/\left<w_d,w_{d'}\right>$, as the following two propositions show.

\begin{prop}\label{quotientell}
    The quotient curve $X_0(N)/\left<w_d,w_{d'}\right>$ is an elliptic curve over $\Q$ for the following values of $N,d,d'$:

\begin{center}
\begin{tabular}{|c|c||c|c||c|c||c|c|}
\hline
$N$ & $(d,d')$ & $N$ & $(d,d')$ & $N$ & $(d,d')$ & $N$ & $(d,d')$\\
    \hline
    
$70$ & $(2,35)$ & $86$ & $(2,43)$ & $96$ & $(3,32)$ & $99$ & $(9,11)$\\
$105$ & $(3,35)$ & $110$ & $(2,5)$ & $111$ & $(3,37)$ & $118$ & $(2,59)$\\ $123$ & $(3,41)$ & $124$ & $(4,31)$ &$141$ & $(3,47)$ & $142$ & $(2,71)$\\ $143$ & $(11,13)$ & $145$ & $(5,29)$ & $155$ & $(5,31)$ & $159$ & $(3,53)$\\
$188$ & $(4,47)$ & & & & & &\\
 
    \hline
\end{tabular}
\end{center}
    
\end{prop}

\begin{prop}\label{quotienthyper}
    The quotient curve $X_0(N)/\left<w_d,w_{d'}\right>$ is a hyperelliptic curve over $\Q$ for the following values of $N,d,d'$. Here $g$ denotes the genus of the curve $X_0(N)/\left<w_d,w_{d'}\right>$.

\begin{center}
\begin{tabular}{|c|c|c||c|c|c||c|c|c|}
\hline
$N$ & $(d,d')$ & $g$ & $N$ & $(d,d')$ & $g$ & $N$ & $(d,d')$ & $g$\\
    \hline
    
$66$ & $(3,22)$ & $2$ & $70$ & $(7,10)$ & $2$ & $78$ & $(2,3)$ & $3$\\
$84$ & $(3,4),(7,12),(4,21),(3,28)$ & $2$ & $88$ & $(8,11)$ & $2$ & $90$ & $(2,45),(5,18),(9,10)$ & $2$\\
$93$ & $(3,31)$ & $2$ & $102$ & $(2,51),(3,17)$ & $2$ & $104$ & $(8,13)$ & $2$\\
$105$ & $(3,5),(3,7),(7,15)$ & $3$ & $106$ & $(2,53)$ & $2$ & $110$ & $(2,5),(2,11),(5,22)$ & $3$\\
$112$ & $(7,16)$ & $2$ & $114$ & $(2,19)$ & $3$ & $114$ & $(3,38)$ & $2$\\
$115$ & $(5,23)$ & $2$ & $116$ & $(4,29)$ & $2$ & $117$ & $(9,13)$ & $2$\\
$120$ & $(8,15),(15,24)$ & $2$ & $120$ & $(5,24)$ & $3$ & $122$ & $(2,61)$ & $2$\\
$126$ & $(2,63),(14,18)$ & $2$ & $126$ & $(7,9),(9,14)$ & $3$ & $129$ & $(3,43)$ & $2$\\
$130$ & $(10,26)$ & $2$ & $130$ & $(2,13)$ & $3$ & $132$ & $(4,11)$ & $2$\\ $133$ & $(7,19)$ & $2$ & $134$ & $(2,67)$ & $2$ & $135$ & $(5,27)$ & $2$\\
$136$ & $(8,17)$ & $3$ & $138$ & $(3,23),(6,23)$ & $2$ & $140$ & $(4,35)$ & $2$\\
$146$ & $(2,73)$ & $2$ & $147$ & $(3,49)$ & $2$ & $150$ & $(6,50)$ & $2$\\
$153$ & $(9,17)$ & $2$ & $156$ & $(4,39)$ & $2$ & $158$ & $(2,79)$ & $2$\\
$161$ & $(7,23)$ & $2$ &$165$ & $(11,15)$ & $3$ & $166$ & $(2,83)$ & $2$\\
$168$ & $(21,24)$ & $4$ & $171$ & $(9,19)$ & $3$ & $176$ & $(11,16)$ & $4$\\
$177$ & $(3,59)$ & $2$ & $184$ & $(8,23)$ & $2$ & $190$ & $(5,19)$ & $2$\\
$195$ & $(5,39)$ & $3$ & $205$ & $(5,41)$ & $2$ & $206$ & $(2,103)$ & $2$\\
$207$ & $(9,23)$ & $3$ &$209$ & $(11,19)$ & $2$ & $213$ & $(3,71)$ & $2$\\
$215$ & $(5,42)$ & $2$ & $221$ & $(13,17)$ & $2$ & $279$ & $(9,31)$ & $5$\\
$284$ & $(4,71)$ & $2$ & $287$ & $(7,41)$ & $2$ & $299$ & $(13,23)$ & $2$\\
 
    \hline
\end{tabular}
\end{center}
    
\end{prop}

\begin{proof}
    Every curve of genus $2$ is hyperelliptic and \cite{FurumotoHasegawa1999} gives us all hyperelliptic quotients of genus $g\geq3$. All quotients of $X_0(N)$ are defined over $\Q$ so all these curves are hyperelliptic curves over $\Q$. 
\end{proof}

If the defining equation of the (elliptic or hyperelliptic) curve $X_0(N)/\left<w_d,w_{d'}\right>$ is $y^2+h(x)y=f(x)$, we can take a degree $2$ rational map to $\PP^1$ to be $y$. Therefore, the desired degree $4$ morphism is \[X_0^{+d}(N)\to X_0(N)/\left<w_d,w_{d'}\right>\xrightarrow{y}\PP^1.\]

\begin{prop}\label{trigonalmap}
    There exists a degree $3$ rational map from $X_0^{+d}(N)$ to $\mathbb{P}^1$ for
$$(N,d)\in\{(66,33),(74,37),(86,43),(108,4),(112,7)\}.$$
\end{prop}

\begin{proof}
    The curve $X_0^{+d}(N)$ is of genus $4$ and we can use the \texttt{Magma} function \texttt{Genus4GonalMap(C)} to get the desired map. For example, the code for the curve $X_0^{+33}(66)$ is:
    \begin{lstlisting}[language=Python]
    
> X:=X0NQuotient(66,[33]);
> assert Genus(X) eq 4;
> Genus4GonalMap(X);
3 Mapping from: Crv: X to Curve over Rational Field defined by
0
with equations :
-x[2]
-x[1] + x[3]

    \end{lstlisting}

    If a genus $4$ curve is not elliptic nor hyperelliptic, then the function \texttt{Genus4GonalMap(C)} returns a degree $3$ map. This map can either be defined over $\Q$ or a quadratic field. In all these cases it will be defined over $\Q$, as shown in the code example.
\end{proof}

\begin{prop}\label{tetragonalmap}
    There exists a degree $4$ rational map from $X_0^{+d}(N)$ to $\mathbb{P}^1$ for
    $$(N,d)\in\{(144,9),(144,16),(148,4),(160,32),(196,4),(208,16),(217,31),\}.$$
\end{prop}

\begin{proof}
    For $(N,d)\in\{(144,9),(144,16),(196,4)\}$ we found a divisor of the form $D=P_1+P_2+P_3+P_4$ with dimension $\ell(D)=2$, where $P_i\in X_0^{+d}(N)(\Q)$.

    In the other four cases we were not able to find a degree $4$ function whose polar divisor is supported on rational points so we had to search for quadratic points.

    We searched for quadratic points by intersecting the curve $X_0^{+d}(N)$ with hyperplanes of the form
    $$b_0x_0+\ldots+b_kx_k=0,$$
    where $b_0,\ldots,b_k\in \Z$ are coprime and chosen up to a certain bound, a similar idea as in \cite[Section 3.2]{Box19}. We can improve this by noting that, in a quadratic point $(x_0,\ldots,x_k)$, already its first three coordinates must be linearly dependent over $\Q$. Therefore, it is enough to check the hyperplanes
    $$b_0x_0+b_1x_1+b_2x_2=0.$$
    In all of these cases we found a divisor of the form $D=P_1+P_2+Q+\sigma(Q)$ with dimension $\ell(D)=2$, where $P_1,P_2\in X_0^{+d}(N)(\Q)$, and $Q$ is one of the quadratic points we found.

    \cite[Appendix B.12]{Stichtenoth09} now tells us that \[L_\Q(D)=\Q(X_0^{+d}(N))\cap L(D)\] is of the same dimension $\ell(D)=2$ over $\Q$. This means that there exists a non-constant rational function in $L_\Q(D)$. As the curve $X_0^{+d}(N)$ is neither hyperelliptic nor trigonal by \cite{FurumotoHasegawa1999, HasegawaShimura1999}, its degree is equal to $4$.
\end{proof}

\begin{comment}

\begin{lem}
    Suppose $C/\Q$ is a curve and $P\in C(\Q)$. If $D\geq0$ is a degree $d$ divisor such that $\ell(D)\geq2$ and if $\gon_\C(C)\geq d$, then $\gon_\Q(C)=d$.
\end{lem}

\begin{proof}
    Since $\ell(D)\geq2$, there is a non-constant function $f\in L(D)$. Let us define $g:=f-f(P)$. It has the same degree as $f$.
    As $\gon_\C(C)\geq d$, we have \[\textup{div }g=P+R-D,\] where $R\geq0$ is some degree $d-1$ divisor. If we now take some $\sigma\in\textup{Gal}_{\overline{\Q}/\Q}$, we get \[\textup{div }\sigma(g)=P+\sigma(R)-D.\] Therefore, the divisor $R-\sigma(R)$ is principal, which is, by the $\C$-gonality bound, only possible if $R=\sigma(R)$. This means that $g$ and $\sigma(g)$ have the same divisor.

    If we look at $g$ just as a rational function and not as an element of the function field of $C$, there exists some rational point $P'$ such that $g(P')\neq0,\infty$ (this point $P'$ is not necessarily on $C$). Let us define $h:=g/g(P')$ so that $h(P')=1$. In the same way as for $g$, we prove that $h$ and $\sigma(h)$ have the same divisor, implying that $\sigma(h)=c\cdot h$ for some $c\in\C$. By plugging in $P'$ we get \[1=\sigma(1)=\sigma(h)(P')=c\cdot h(P')=c\cdot1=c.\]

    So $h=\sigma(h)$ for every $\sigma\in\textup{Gal}_{\overline{\Q}/\Q}$, meaning that $h$ is a degree $d$ rational function.
\end{proof}
    
\end{comment}

\section{Lower bounds on $\C$-gonality}\label{C-gonality_section}
In this section we will prove that the remaining curves $X_0^{+d}(N)$ are not $\C$-tetragonal. The Castelnuovo-Severi inequality is one tool to do that

\begin{prop}[Castelnuovo-Severi inequality]\label{tm:CS}
Let $k$ be a perfect field, and let $X,\ Y, \ Z$ be curves over $k$. Let non-constant morphisms $\pi_Y:X\rightarrow Y$ and $\pi_Z:X\rightarrow Z$ over $k$ be given, and let their degrees be $m$ and $n$, respectively. Assume that there is no morphism $X\rightarrow X'$ of degree $>1$ through which both $\pi_Y$ and $\pi_Z$ factor. Then the following inequality holds:
\begin{equation} \label{eq:CS}
g(X)\leq m \cdot g(Y)+n\cdot g(Z) +(m-1)(n-1).
\end{equation}
\end{prop}

Since $\C$ and $\Q$ are both perfect fields, we can use the Castelnuovo-Severi inequality to get lower bounds on both $\C$ and $\Q$-gonalities.

In this statement of Castelnuovo-Severi inequality the hypothetical morphism $X\to X'$ is defined over $\overline{k}$. However, in \cite[Theorem 12]{KhawajaSiksek2023} it was recently proved that we can suppose that this morphism is defined over $k$. This is not important to us because we are using the inequality to obtain bounds on $\C$-gonality. However, if one wishes to obtain bounds on $\Q$-gonality, the result from \cite{KhawajaSiksek2023} is helpful.

\begin{prop}\label{CSprop}
The $\C$-gonality of the curve $X_0^{+d}(N)$ is at least $5$ for the following values of $N$ and $d$. Here $g$ denotes the genus of the curve $X_0^{+d}(N)$ and $g'$ denotes the genus of the quotient curve $X_0(N)/\left<w_d,w_{d'}\right>$.
\begin{center}
\begin{longtable}{|c|c|c|c||c|c|c|c||c|c|c|c|}

\hline
\addtocounter{table}{-1}
$(N,d)$ & $g$ & $d'$ & $g'$ & $(N,d)$ & $g$ & $d'$ & $g'$ & $(N,d)$ & $g$ & $d'$ & $g'$\\
    \hline
    
$(132,3)$ & $10$ & $44$ & $3$ & $(132,12)$ & $10$ & $11$ & $3$ & $(138,2)$ & $11$ & $23$ & $3$\\
$(138,46)$ & $11$ & $2$ & $3$ & $(140,7)$ & $10$ & $20$ & $3$ & $(150,2)$ & $10$ & $75$ & $3$\\
$(150,3)$ & $10$ & $50$ & $3$ & $(156,3)$ & $11$ & $13$ & $3$ & $(156,12)$ & $11$ & $52$ & $3$\\
$(156,13)$ & $12$ & $3$ & $3$ & $(156,52)$ & $12$ & $12$ & $3$ & $(174,2)$ & $14$ & $87$ & $3$\\
$(174,3)$ & $14$ & $29$ & $3$ & $(174,6)$ & $13$ & $58$ & $4$ & $(174,29)$ & $13$ & $3$ & $3$\\
$(174,58)$ & $14$ & $6$ & $4$ & $(182,2)$ & $13$ & $91$ & $4$ & $(182,7)$ & $13$ & $26$ & $4$\\
$(182,13)$ & $12$ & $14$ & $4$ & $(182,14)$ & $11$ & $26$ & $3$ & $(182,26)$ & $10$ & $14$ & $3$\\
$(182,91)$ & $10$ & $14$ & $3$ & $(183,61)$ & $10$ & $3$ & $3$ & $(186,3)$ & $14$ & $62$ & $4$\\
$(186,6)$ & $14$ & $62$ & $5$ & $(186,62)$ & $14$ & $6$ & $5$ & $(190,2)$ & $14$ & $95$ & $3$\\
$(190,38)$ & $14$ & $10$ & $3$ & $(195,3)$ & $13$ & $65$ & $3$ & $(195,15)$ & $13$ & $39$ & $3$\\
$(198,2)$ & $14$ & $99$ & $5$ & $(198,9)$ & $15$ & $11$ & $5$ & $(204,3)$ & $16$ & $68$ & $5$\\
$(204,4)$ & $15$ & $51$ & $5$ & $(204,12)$ & $16$ & $51$ & $5$ & $(204,17)$ & $16$ & $4$ & $6$\\
$(210,2)$ & $21$ & $35$ & $8$ & $(210,3)$ & $21$ & $35$ & $7$ & $(210,5)$ & $19$ & $7$ & $7$\\
$(210,6)$ & $19$ & $35$ & $6$ & $(210,7)$ & $21$ & $5$ & $7$ & $(210,10)$ & $21$ & $14$ & $6$\\
$(210,14)$ & $16$ & $10$ & $6$ & $(210,15)$ & $21$ & $21$ & $7$ & $(210,21)$ & $19$ & $15$ & $7$\\
$(210,30)$ & $21$ & $35$ & $8$ & $(210,42)$ & $21$ & $35$ & $8$ & $(210,70)$ & $21$ & $2$ & $8$\\
$(210,105)$ & $19$ & $3$ & $7$ & $(220,4)$ & $16$ & $55$ & $4$ & $(220,5)$ & $16$ & $11$ & $3$\\
$(220,11)$ & $13$ & $5$ & $4$ & $(220,20)$ & $16$ & $44$ & $4$ & $(220,44)$ & $13$ & $20$ & $4$\\
$(222,2)$ & $18$ & $111$ & $4$ & $(222,3)$ & $17$ & $37$ & $5$ & $(222,6)$ & $18$ & $74$ & $3$\\
$(222,37)$ & $18$ & $3$ & $5$ & $(222,74)$ & $13$ & $6$ & $3$ & $(222,111)$ & $10$ & $6$ & $3$\\
$(230,2)$ & $17$ & $115$ & $5$ & $(230,5)$ & $16$ & $46$ & $5$ & $(230,10)$ & $16$ & $23$ & $6$\\
$(230,23)$ & $17$ & $10$ & $6$ & $(230,46)$ & $15$ & $5$ & $5$ & $(230,115)$ & $14$ & $2$ & $5$\\
$(231,3)$ & $15$ & $77$ & $3$ & $(231,7)$ & $15$ & $33$ & $4$ & $(231,11)$ & $15$ & $21$ & $4$\\
$(231,21)$ & $13$ & $33$ & $4$ & $(231,33)$ & $13$ & $21$ & $4$ & $(231,77)$ & $11$ & $3$ & $3$\\
$(234,18)$ & $18$ & $117$ & $7$ & $(234,9)$ & $17$ & $26$ & $6$ & $(234,18)$ & $18$ & $26$ & $7$\\
$(238,2)$ & $17$ & $119$ & $3$ & $(238,7)$ & $17$ & $17$ & $3$ & $(238,14)$ & $17$ & $34$ & $3$\\
$(238,17)$ & $15$ & $7$ & $3$ & $(238,34)$ & $15$ & $14$ & $3$ & $(245,5)$ & $10$ & $49$ & $3$\\
$(246,2)$ & $19$ & $123$ & $7$ & $(246,3)$ & $20$ & $41$ & $7$ & $(246,6)$ & $20$ & $41$ & $7$\\
$(246,82)$ & $20$ & $2$ & $8$ & $(248,8)$ & $15$ & $31$ & $3$ & $(249,3)$ & $14$ & $83$ & $3$\\
$(250,2)$ & $14$ & $125$ & $5$ & $(252,4)$ & $17$ & $63$ & $5$ & $(252,7)$ & $19$ & $9$ & $7$\\
$(252,9)$ & $19$ & $7$ & $7$ & $(252,28)$ & $19$ & $36$ & $7$ & $(252,36)$ & $19$ & $28$ & $7$\\
$(254,2)$ & $16$ & $127$ & $4$ & $(258,2)$ & $20$ & $43$ & $8$ & $(258,3)$ & $20$ & $86$ & $7$\\
$(258,6)$ & $21$ & $86$ & $7$ & $(258,43)$ & $21$ & $2$ & $8$ & $(258,129)$ & $18$ & $6$ & $7$\\
$(259,37)$ & $12$ & $7$ & $4$ & $(261,9)$ & $13$ & $29$ & $4$ & $(262,2)$ & $16$ & $131$ & $4$\\
$(267,3)$ & $15$ & $89$ & $4$ & $(270,2)$ & $22$ & $135$ & $7$ & $(270,5)$ & $21$ & $27$ & $8$\\
$(270,10)$ & $22$ & $54$ & $7$ & $(270,27)$ & $22$ & $5$ & $8$ & $(270,54)$ & $19$ & $10$ & $7$\\
$(272,17)$ & $16$ & $16$ & $6$ & $(274,2)$ & $16$ & $137$ & $6$ & $(275,25)$ & $13$ & $11$ & $4$\\
$(278,2)$ & $17$ & $139$ & $5$ & $(291,97)$ & $16$ & $3$ & $6$ & $(295,5)$ & $15$ & $59$ & $3$\\
$(297,27)$ & $16$ & $11$ & $6$ & $(298,2)$ & $18$ & $149$ & $7$ & $(300,4)$ & $19$ & $75$ & $7$\\
$(300,75)$ & $19$ & $4$ & $7$ & $(302,2)$ & $19$ & $151$ & $5$ & $(303,3)$ & $17$ & $101$ & $3$\\
$(303,101)$ & $10$ & $3$ & $3$ & $(305,5)$ & $14$ & $61$ & $4$ & $(305,61)$ & $12$ & $5$ & $4$\\
$(319,11)$ & $15$ & $29$ & $4$ & $(319,29)$ & $12$ & $11$ & $4$ & $(321,3)$ & $18$ & $107$ & $4$\\
$(321,107)$ & $12$ & $3$ & $4$ & $(323,17)$ & $15$ & $19$ & $5$ & $(329,7)$ & $16$ & $47$ & $3$\\
$(329,47)$ & $11$ & $7$ & $3$ & $(335,5)$ & $16$ & $67$ & $4$ & $(335,67)$ & $17$ & $5$ & $4$\\
$(341,11)$ & $14$ & $31$ & $4$ & $(341,31)$ & $16$ & $11$ & $4$ & $(355,5)$ & $18$ & $71$ & $4$\\
$(371,7)$ & $17$ & $53$ & $5$ & $(371,53)$ & $18$ & $7$ & $5$ & $(377,13)$ & $16$ & $29$ & $5$\\
$(377,29)$ & $14$ & $13$ & $5$ & $(391,17)$ & $16$ & $23$ & $5$ & $(391,23)$ & $18$ & $17$ & $5$\\
 
    \hline

\caption*{}
\end{longtable}
\end{center}
\end{prop}

\begin{proof}
    The results of \cite{FurumotoHasegawa1999} and \cite{HasegawaShimura1999} tell us that these curves $X_0^{+d}(N)$ are not hyperelliptic nor trigonal over $\C$. Suppose there exists a degree $4$ map from $X_0^{+d}(N)$ to $\PP^1$. We apply the Castelnuovo-Severi inequality with $f$ and a degree $2$ quotient map $\pi: X_0^{+d}(N)\to X_0(N)/\left<w_d,w_{d'}\right>$. 
    
    Since $g(X_0^{+d}(N))$ is too high in all cases here, $f$ has to factor through $\pi$ (because $\deg\pi=2$ is a prime). In that case we would have $\gon_\C(X_0(N)/\left<w_d,w_{d'}\right>)=2$ and the quotient curve $X_0(N)/\left<w_d,w_{d'}\right>$ would need to be elliptic or hyperelliptic. However, we can again use \cite{FurumotoHasegawa1999} to eliminate this possibility.
\end{proof}
The genera of all these curves can be obtained using Philippe Michaud Jacobs's function \texttt{genus\_quo} which is available on Github. A code for the curve $X_0(132)/\left<w_3,w_{44}\right>$ is

    \begin{lstlisting}[language=Python]
    
> load "new_models.m";
> genus_quo(132,[3,44]);
3 

    \end{lstlisting}

Another helpful tool here is the Tower theorem \cite[Theorem 2.1]{NguyenSaito}. We use its corollary \cite[Corollary 4.6.]{NajmanOrlic22} which says that for curves of genus $\geq10$, the existence of a degree $4$ map to $\mathbb{P}^1$ over $\C$ is equivalent with the existence of a rational degree $4$ map to $\mathbb{P}^1$.

In order to bound the number of levels we need to check, we can use the theorem by Kim and Sarnak, mentioned in the Introduction.

\begin{thm}[Appendix 2 to \cite{Kim2002}]
    Let $X_\Gamma$ be the algebraic curve corresponding to a congruence subgroup $\Gamma\subseteq \SL_2(\Z)$ of index
    $$D_\Gamma=[\SL_2(\Z):\pm\Gamma].$$
    If $X_\Gamma$ is $d$-gonal, then $D_\Gamma\leq \frac{12000}{119}d$.
\end{thm}

\begin{cor}
    The curve $X_0^{+d}(N)$ is not $\C$-tetragonal for $N\geq807$.
\end{cor}

\begin{proof}
    Suppose $X_0^{+d}(N)$ is $\C$-tetragonal. Then $X_0(N)$ has a degree $8$ map to $\mathbb{P}^1$. Since $-I\in\Gamma_0(N)$, we have that $\psi(N)=D_{\Gamma_0(N)}\leq\frac{12000}{119}\cdot8$ (here $\psi(N)=N\prod_{q\mid N}(1+\frac{1}{q})$, as mentioned in \Cref{lemmaogg}).
\end{proof}

In \Cref{Fpsection} we determined the $\Q$-gonality of curves $X_0^{+d}(N)$. For those curves of genus $g\geq10$ we can use the Tower theorem to determine the existence of a degree $4$ map to $\PP^1$ over $\C$, as mentioned at the beginning of this section.

We used \texttt{Magma} to list all curves $X_0^{+d}(N)$ with $d<N<807$ with genus $5\leq g\leq 9$ (those of genus $g\leq4$ surely have $\gon_\Q\leq4$ by \Cref{poonen}(iv)). Moreover, in \Cref{degree4mapsection} and \Cref{CSprop} some of these curves were solved. 

This leaves us with reasonably many cases that are not yet solved. The only pairs $(N,d)$ we need to check are in the table below. One can check in the Tables \ref{tab:main1}, \ref{tab:main2}, \ref{tab:main3} that the $\C$-gonality of all the other curves has been bounded in \Cref{CSprop} or in Sections \ref{Fpsection} and \ref{degree4mapsection}.

\begin{table}[ht]
\centering
\begin{tabular}{|c|c||c|c||c|c||c|c|}
\hline
$(N,d)$ & $g(X_0^{+d}(N))$ & $(N,d)$ & $g(X_0^{+d}(N))$ & $(N,d)$ & $g(X_0^{+d}(N))$ & $(N,d)$ & $g(X_0^{+d}(N))$\\
    \hline

    $(102,6)$ & $8$ & $(102,34)$ & $8$ & $(114,6)$ & $9$ & $(114,57)$ & $8$\\
    $(120,3)$ & $9$ & $(130,5)$ & $9$ & $(140,20)$ & $8$ & $(148,37)$ & $9$\\
    $(152,8)$ & $8$ & $(152,19)$ & $9$ & $(154,7)$ & $9$ & $(154,77)$ & $9$\\
    $(160,5)$ & $9$ & $(162,2)$ & $8$ & $(162,81)$ & $7$ & $(164,4)$ & $9$\\
    $(172,43)$ & $9$ & $(174,87)$ & $8$ & $(175,7)$ & $8$ & $(175,25)$ & $8$\\
    $(178,89)$ & $8$ & $(183,3)$ & $9$ & $(185,5)$ & $9$ & $(185,37)$ & $9$\\
    $(187,11)$ & $9$ & $(187,17)$ & $7$ & $(196,49)$ & $9$ & $(202,101)$ & $9$\\
    $(203,7)$ & $9$ & $(214,107)$ & $9$ & $(225,9)$ & $9$ & $(238,119)$ & $7$\\
    $(245,49)$ & $9$ & $(248,31)$ & $9$ & $(249,83)$ & $8$ & $(262,131)$ & $9$\\
    $(267,89)$ & $9$ & $(295,59)$ & $9$ & & & &\\

    \hline
\end{tabular}
\vspace{5mm}
\caption{Curves $X_0^{+d}(N)$ of level $N\leq806$ and genus at most $9$ that are not $\Q$-tetragonal.}
\label{tablebetti}
\end{table}

We will use graded Betti numbers $\beta_{i,j}$ to disprove the existence of a degree $4$ morphism to $\mathbb{P}^1$. We will follow the notation in \cite[Section 1.]{JeonPark05}. The results we mention can be found there and in \cite[Section 3.1.]{NajmanOrlic22}.

\begin{definition}
    For a curve $X$ and divisor $D$ of degree $d$, $g_d^r$ is a subspace of the Riemann-Roch space $L(D)$ such that $\dim V=r+1$.
\end{definition}

Therefore, we want to determine whether $X_0^{+d}(N)$ has a $g_4^1$. Green's conjecture relates graded Betti numbers $\beta_{i,j}$ with the existence of $g_d^r$.

\begin{conj}[Green, \cite{Green84}]
    Let $X$ be a curve of genus $g$. Then $\beta_{p,2}\neq 0$ if and only if there exists a divisor $D$ on $X$ of degree $d$ such that a subspace $g_d^r$ of $L(D)$ satisfies $d\leq g-1$, $r=\ell(D)-1\geq1$, and $d-2r\leq p$.
\end{conj}

The "if" part of this conjecture was proved in the same paper.

\begin{thm}[Green and Lazarsfeld, Appendix to \cite{Green84}]
    Let $X$ be a curve of genus $q$. If $\beta_{p,2}=0$, then there does not exist a divisor $D$ on $X$ of degree $d$ such that a subspace $g_d^r$ of $L(D)$ satisfies $d\leq g-1$, $r\geq1$, and $d-2r\leq p$.
\end{thm}

\begin{cor}
    Let $X$ be a curve of genus $g\geq5$ with $\beta_{2,2}=0$. Then $\textup{gon}_\C(X)\geq5$.
\end{cor}

\begin{cor}\label{bettiprop}
    The curve $X_0^{+d}(N)$ is not tetragonal for all $(N,d)$ in \Cref{tablebetti}.
\end{cor}

\begin{proof}
    For all these curves we compute $\beta_{2,2}=0$. We present the \texttt{Magma} code for the curve $X_0^{+6}(102)$.

    \begin{lstlisting}[language=Python]

> X:=X0NQuotient(102,[6]);
> A:=QuotientModule(DefiningIdeal(X));
> BettiTable(A);
[
    [ 1, 0, 0, 0, 0, 0, 0 ],
    [ 0, 15, 35, 21, 0, 0, 0 ],
    [ 0, 0, 0, 21, 35, 15, 0 ],
    [ 0, 0, 0, 0, 0, 0, 1 ]
]
0
> BettiNumber(A,2,4);
0

    \end{lstlisting}

    The function \texttt{BettiTable()} returns a Betti table $S$ and a shift $s$. This is designed so that if $A$ is non-zero, then S[1, 1] is always non-zero and S[i, j] equals \texttt{BettiNumber}(A, i-1, (i-1)+(j-1)+s). (So the degrees are shifted by s.) \cite{magma} We need to subtract $1$ since the indexation in the table starts with $1$ instead of $0$. In the given example we have $s=0$ and
    
    \[\texttt{BettiNumber(A,2,4)} = \texttt{BettiTable(A)[3][3]} = 0,\]  \[\texttt{BettiNumber(A,1,2)} = \texttt{BettiTable(A)[2][2]} = 15.\]

\end{proof}

\begin{remark}
    If the reader wishes to further check these calculations, \cite[Table 1]{Schreyer1986} gives all possible Betti tables for curves of genus $g\leq8$. The genus $8$ curve $X_0^{+6}(102)$ would fall under the general genus $8$ case in that table.

    Notice that there the Betti numbers are indexed differently than in \cite{JeonPark05} and this paper. For example, Green's conjecture in \cite{Schreyer1986} is stated differently. The indexation there matches with the indexation in the \texttt{Magma} function \texttt{BettiNumber} (our $\beta_{2,2}$ is there indexed as $\beta_{2,4}$).
\end{remark}

\section{Proofs of the main theorems}

\begin{proof}[Proof of \Cref{trigonalthm}]
    Hasegawa and Shimura \cite{HasegawaShimura1999} already solved the cases when $g(X_0^+(N))\neq4$. \Cref{Fp_trigonal} and \Cref{trigonalmap} solve the cases when the genus is equal to $4$.
\end{proof}

\begin{proof}[Proof of \Cref{tetragonalthm}]
    We can suppose that the genus of the curve $X_0^{+d}(N)$ is at least $4$, otherwise the $\Q$-gonality is at most $3$ due to \Cref{poonen}.

    The results of \cite{FurumotoHasegawa1999} give us all hyperelliptic quotients of $X_0(N)$ and the results of \cite{HasegawaShimura1999} give us all $\C$-trigonal curves $X_0^{+d}(N)$. There are exactly $8$ cases when the curve $X_0^{+d}(N)$ is $\C$-trigonal of genus $g\geq5$, namely
    $$(N,d)\in\{(117,13),(122,122),(146,146),(147,3),(162,162),(164,164),(181,181),(227,227)\},$$
    and in these cases the Tower theorem implies that the $\Q$-gonality is also equal to $3$.

    For genus $4$ curves listed in the statement of the theorem, we used \Cref{Fp_trigonal} to prove that there are no degree $3$ rational maps to $\mathbb{P}^1$. Therefore, the $\Q$-gonality of these curves must be equal to $4$.

    We can now suppose that the curve $X_0^{+d}(N)$ is of genus $g\geq5$ and is not hyperelliptic nor trigonal over $\C$. For the curves listed in the theorem, in \Cref{degree4mapsection} we find a rational degree $4$ map to $\mathbb{P}^1$. In the remaining cases, we prove in Sections \ref{Fpsection} and \ref{C-gonality_section} that there are no degree $4$ maps to $\mathbb{P}^1$, and so $\textup{gon}_\C X_0^{+d}(N)>4$ in these cases.
\end{proof}

We summarize the results of this paper in the following table. For each value of $N$, we give links to all results used to solve the curves $X_0^{+d}(N)$ for $d\neq N$. We skip the curves of genus at most $3$ and in the table we write $g\leq3$ when we want to say that all curves $X_0^{+d}(N)$ for that level $N$ are of genus $g\leq3$. We also skip the levels $N$ that are prime powers.

\clearpage
\begin{table}[ht]
\centering
\begin{tabular}{|c|c||c|c|}
  \hline
  % after \\: \hline or \cline{col1-col2} \cline{col3-col4} ...
  $N$ & Results used & $N$ & Results used\\
  \hline

$\leq59$ & $g\leq3$ & $60$ & $g\leq3$ for $d=4,15,20$, \ref{Fp_trigonal}, \cite{FurumotoHasegawa1999} for $d=12$\\
$61$ & & $62$ & $g\leq3$ for $d=31$, {\cite{FurumotoHasegawa1999}} for $d=2$\\
$63$ & $g\leq3$ & $64$ & \\
$65$ & $g\leq3$ & $66$ & \ref{Fp_trigonal}, \ref{quotienthyper}, \ref{trigonalmap}, \cite{FurumotoHasegawa1999} for $d=6$\\
$67$ & & $68$ & $g\leq3$ for $d=4$, \ref{Fp_trigonal}\\
$69$ & $g\leq3$ for $d=23$, \cite{FurumotoHasegawa1999} for $d=3$ & $70$ & $g\leq3$ for $d=14,35$, \ref{Fp_trigonal}, \ref{quotientell}, \ref{quotienthyper}, \cite{FurumotoHasegawa1999} for $d=10$\\
$71$ & & $72$ & $g\leq3$\\
$73$ & & $74$ & \ref{Fp_trigonal}, \ref{trigonalmap}\\
$75$ & $g\leq3$ & $76$ & $g\leq3$ for $d=19$, \ref{Fp_trigonal}\\
$77$ & $g\leq3$ for $d=7$, \ref{Fp_trigonal} & $78$ & \ref{quotienthyper}, \cite{FurumotoHasegawa1999} for $d=26$\\
$79$ & & $80$ & $g\leq3$ for $d=16$, \ref{Fp_trigonal}\\
$81$ & & $82$ & $g\leq3$ for $d=41$, \ref{Fp_trigonal}\\
$83$ & & $84$ & \ref{quotienthyper}\\
$85$ & \ref{Fp_trigonal} & $86$ & \ref{quotientell}, \ref{trigonalmap}\\
$87$ & \cite{FurumotoHasegawa1999} & $88$ & \ref{Fp_trigonal}, \ref{quotienthyper}\\
$89$ & & $90$ & \ref{quotienthyper}\\
$91$ & $g\leq3$ for $d=13$, \ref{Fp_trigonal} & $92$ & \cite{FurumotoHasegawa1999}\\
$93$ & \ref{Fp_trigonal}, \ref{quotienthyper} & $94$ & \cite{FurumotoHasegawa1999}\\
$95$ & \cite{FurumotoHasegawa1999} & $96$ & $g\leq3$ for $d=32$, \ref{quotientell}\\
$97$ & & $98$ & $g\leq3$ for $d=49$, \ref{Fp_trigonal}\\
$99$ & $g\leq3$ for $d=11$, \ref{quotientell} & $100$ & $g\leq3$ for $d=4$, \ref{Fp_trigonal}\\
$101$ & & $102$ & \ref{quotienthyper}, \ref{bettiprop}\\
$103$ & & $104$ & \ref{quotienthyper}\\
$105$ & \ref{quotientell}, \ref{quotienthyper} & $106$ & \ref{quotienthyper}\\
$107$ & & $108$ & \ref{Fp_trigonal}, \ref{trigonalmap}\\
$109$ & & $110$ & \ref{Fp_trigonal}, \ref{quotientell}, \ref{quotienthyper}\\
$111$ & \ref{quotientell} & $112$ & \ref{quotienthyper}, \ref{tetragonalmap}\\
$113$ & & $114$ & \ref{quotienthyper}, \ref{bettiprop}\\
$115$ & \ref{quotienthyper} & $116$ & \ref{quotienthyper}\\
$117$ & \ref{quotienthyper}, \cite{HasegawaShimura1999} for $d=13$ & $118$ & \ref{quotientell}\\
$119$ & \cite{FurumotoHasegawa1999} & $120$ & \ref{quotienthyper}, \ref{bettiprop}\\
$121$ & & $122$ & \ref{quotienthyper}\\
$123$ & \ref{quotientell} & $124$ & \ref{quotientell}\\
$125$ & & $126$ & \ref{quotienthyper}\\
$127$ & & $128$ & \\
$129$ & \ref{quotienthyper} & $130$ & \ref{quotienthyper}, \ref{bettiprop}\\
$131$ & & $132$ & \ref{Fp_gonality}, \ref{CSprop}, \ref{quotienthyper}\\
$133$ & \ref{Fp_trigonal}, \ref{quotienthyper} & $134$ & \ref{quotienthyper}\\
$135$ & \ref{quotienthyper} & $136$ & \ref{quotienthyper}\\
$137$ & & $138$ & \ref{CSprop}, \ref{quotienthyper}\\
$139$ & & $140$ & \ref{Fp2points_prop_quotient}, \ref{CSprop}, \ref{quotienthyper}, \ref{bettiprop}\\

  \hline
\end{tabular}\\
\vspace{5mm}
\caption{Methods used for $N\leq140$.}
\label{tab:main1}
\end{table}

\clearpage

\begin{table}[ht]
\centering
\begin{tabular}{|c|c||c|c||c|c||c|c|}
  \hline
  % after \\: \hline or \cline{col1-col2} \cline{col3-col4} ...
  $N$ & Results used & $N$ & Results used & $N$ & Results used & $N$ & Results used\\
  \hline

$141$ & \ref{quotientell} & $142$ & \ref{quotientell} & $143$ & \ref{quotientell} & $144$ & \ref{tetragonalmap}\\
$145$ & \ref{Fp_trigonal}, \ref{quotientell} & $146$ & \ref{quotienthyper} & $147$ & \ref{quotienthyper}, \cite{HasegawaShimura1999} for $d=3$ & $148$ & \ref{tetragonalmap}, \ref{bettiprop}\\
$149$ & & $150$ & \ref{Fp_gonality}, \ref{CSprop}, \ref{quotienthyper} & $151$ & & $152$ & \ref{bettiprop}\\
$153$ & \ref{quotienthyper} & $154$ & \ref{Fp_gonality}, \ref{bettiprop} & $155$ & \ref{quotientell} & $156$ & \ref{CSprop}, \ref{quotienthyper}\\
$157$ & & $158$ & \ref{quotienthyper} &$159$ & \ref{quotientell} & $160$ & \ref{Fp_gonality}, \ref{tetragonalmap}\\
$161$ & \ref{quotienthyper} & $162$ & \ref{bettiprop} & $163$ & & $164$ & \ref{Fp_gonality}, \ref{bettiprop}\\
$165$ & \ref{Fp2points_prop_quotient}, \ref{Fp_gonality}, \ref{quotienthyper} & $166$ & \ref{quotienthyper} & $167$ & & $168$ & \ref{Fp_gonality}, \ref{quotienthyper}\\
$169$ & & $170$ & \ref{Fp_gonality} & $171$ & \ref{quotienthyper} & $172$ & \ref{Fp_gonality}, \ref{bettiprop}\\
$173$ & & $174$ & \ref{CSprop}, \ref{bettiprop} & $175$ & \ref{bettiprop} & $176$ & \ref{quotienthyper}\\
$177$ & \ref{Fp_trigonal}, \ref{quotienthyper} & $178$ & \ref{CSprop}, \ref{bettiprop} & $179$ & & $180$ & \ref{Fp_gonality}\\
$181$ & & $182$ & \ref{CSprop} & $183$ & \ref{CSprop}, \ref{bettiprop} & $184$ & \ref{quotienthyper}\\
$185$ & \ref{Fp_gonality}, \ref{bettiprop} & $186$ & \ref{Fp_gonality}, \ref{CSprop} & $187$ & \ref{bettiprop} & $188$ & \ref{Fp_trigonal}, \ref{quotientell}\\
$189$ & \ref{CSprop}, \ref{bettiprop} & $190$ & \ref{CSprop}, \ref{quotienthyper} & $191$ & & $192$ & \ref{Fp_gonality}\\
$193$ & & $194$ & \ref{CSprop} & $195$ & \ref{Fp2points_prop_quotient}, \ref{CSprop}, \ref{quotienthyper}, \ref{bettiprop} & $196$ & \ref{tetragonalmap}, \ref{bettiprop}\\
$197$ & & $198$ & \ref{Fp_gonality}, \ref{CSprop} & $199$ & & $200$ & \ref{Fp2points_prop_quotient}, \ref{Fp_gonality}\\
$201$ & \ref{Fp_gonality} & $202$ & \ref{CSprop}, \ref{bettiprop} & $203$ & \ref{CSprop}, \ref{bettiprop} & $204$ & \ref{Fp_gonality}, \ref{CSprop}\\
$205$ & \ref{quotienthyper} & $206$ & \ref{quotienthyper} & $207$ & \ref{quotienthyper} & $208$ & \ref{Fp2points_prop_quotient}, \ref{tetragonalmap}\\
$209$ & \ref{quotienthyper} & $210$ & \ref{Fp_gonality}, \ref{CSprop} & $211$ & & $212$ & \ref{Fp2points_prop_quotient}, \ref{Fp_gonality}\\
$213$ & \ref{quotienthyper} & $214$ & \ref{CSprop}, \ref{bettiprop} & $215$ & \ref{quotienthyper} & $216$ & \ref{Fp_gonality}\\
$217$ & \ref{CSprop}, \ref{tetragonalmap} & $218$ & \ref{Fp_gonality} & $219$ & \ref{Fp_gonality} & $220$ & \ref{Fp2points_prop_quotient}, \ref{CSprop}\\
$221$ & \ref{quotienthyper} & $222$ & \ref{CSprop} & $223$ & & $224$ & \ref{Fp2points_prop_quotient}, \ref{CSprop}\\
$225$ & \ref{Fp2points_prop_quotient}, \ref{bettiprop} & $226$ & \ref{Fp2points_prop_quotient}, \ref{Fp_gonality} & $227$ & & $228$ & \ref{Fp_gonality_quotient}\\
$229$ & & $230$ & \ref{CSprop} & $231$ & \ref{CSprop} & $232$ & \ref{Fp_gonality}\\
$233$ & & $234$ & \ref{Fp_gonality}, \ref{CSprop} & $235$ & \ref{Fp_gonality} & $236$ & \ref{CSprop}\\
$237$ & \ref{Fp2points_prop_quotient}, \ref{Fp_gonality} & $238$ & \ref{CSprop}, \ref{bettiprop} & $239$ & & $240$ & \ref{Fp_gonality}, \ref{Fp_gonality_quotient}\\
$241$ & & $242$ & \ref{Fp2points_prop_quotient} & $243$ & & $244$ & \ref{Fp_gonality}\\
$245$ & \ref{CSprop}, \ref{bettiprop} & $246$ & \ref{Fp_gonality_quotient}, \ref{CSprop} & $247$ & \ref{Fp_gonality} & $248$ & \ref{CSprop}, \ref{bettiprop}\\
$249$ & \ref{CSprop}, \ref{bettiprop} & $250$ & \ref{Fp_gonality}, \ref{CSprop} & $251$ & & $252$ & \ref{Fp_gonality}, \ref{CSprop}\\
$253$ & \ref{Fp_gonality} & $254$ & \ref{Fp2points_prop_quotient}, \ref{CSprop} & $255$ & \ref{Fp2points} & $256$ & \\
$257$ & & $258$ & \ref{Fp_gonality}, \ref{CSprop} & $259$ & \ref{Fp2points_prop_quotient}, \ref{CSprop} & $260$ & \ref{Fp2points}\\
$261$ & \ref{Fp2points_prop_quotient}, \ref{CSprop} & $262$ & \ref{CSprop}, \ref{bettiprop} & $263$ & & $264$ & \ref{Fp_gonality_quotient}\\
$265$ & \ref{Fp_gonality} & $266$ & \ref{Fp2points} & $267$ & \ref{CSprop}, \ref{bettiprop} & $268$ & \ref{Fp2points_prop_quotient}, \ref{Fp_gonality}\\
$269$ & & $270$ & \ref{Fp_gonality_quotient}, \ref{CSprop} & $271$ & & $272$ & \ref{Fp2points_prop_quotient}, \ref{CSprop}\\
$273$ & \ref{Fp_gonality}, \ref{Fp_gonality_quotient} & $274$ & \ref{Fp2points_prop_quotient}, \ref{CSprop} & $275$ & \ref{Fp2points_prop_quotient}, \ref{CSprop} & $276$ & \ref{Fp2points}\\
$277$ & & $278$ & \ref{Fp2points_prop_quotient}, \ref{CSprop} & $279$ & \ref{quotienthyper} & $280$ & \ref{Fp2points}\\
$281$ & & $282$ & \ref{Fp2points} & $283$ & & $284$ & \ref{quotienthyper}\\
$285$ & \ref{Fp2points} & $286$ & \ref{Fp2points} & $287$ & \ref{quotienthyper} & $288$ & \ref{Fp2points_prop_quotient}, \ref{Fp_gonality}\\
$289$ & & $290$ & \ref{Fp2points} & $291$ & \ref{Fp_gonality}, \ref{CSprop} & $292$ & \ref{Fp2points}\\
$293$ & & $294$ & \ref{ogg} & $295$ & \ref{CSprop}, \ref{bettiprop} & $296$ & \ref{Fp2points}\\
$297$ & \ref{Fp2points_prop_quotient}, \ref{CSprop} & $298$ & \ref{Fp2points_prop_quotient}, \ref{CSprop} & $299$ & \ref{quotienthyper} & $300$ & \ref{Fp_gonality_quotient}, \ref{CSprop}\\
    
  \hline
\end{tabular}\\
\vspace{5mm}
\caption{Methods used for $141\leq N\leq300$.}
\label{tab:main2}
\end{table}

\clearpage

\begin{table}[ht]
\centering
\begin{tabular}{|c|c||c|c||c|c||c|c|}
  \hline
  % after \\: \hline or \cline{col1-col2} \cline{col3-col4} ...
  $N$ & Results used & $N$ & Results used & $N$ & Results used & $N$ & Results used\\
  \hline

$301$ & \ref{Fp2points_prop_quotient}, \ref{Fp_gonality} & $302$ & \ref{Fp2points_prop_quotient}, \ref{CSprop} & $303$ & \ref{CSprop} & $304$ & \ref{Fp2points}\\
$305$ & \ref{CSprop} & $309$ & \ref{Fp_gonality_quotient} & $319$ & \ref{CSprop} & $321$ & \ref{CSprop}\\
$323$ & \ref{Fp2points_prop_quotient}, \ref{CSprop} & $325$ & \ref{Fp2points_prop_quotient} & $329$ & \ref{CSprop} & $335$ & \ref{CSprop}\\
$341$ & \ref{CSprop} & $355$ & \ref{Fp2points_prop_quotient}, \ref{CSprop} & $371$ & \ref{CSprop} & $377$ & \ref{CSprop}\\
$391$ & \ref{CSprop} & $420$ & \ref{Fp2points_prop} & & & & \\

  \hline
\end{tabular}\\
\vspace{5mm}
\caption{Methods used for $N\geq300$. Levels $N$ eliminated by \Cref{Fp2points} and prime powers are omitted.}
\label{tab:main3}
\end{table}

\bibliographystyle{siam}
\bibliography{bibliography1}

\end{document}